\documentclass[a4paper]{amsart}
\usepackage[UKenglish]{babel}
\usepackage[T2A]{fontenc}

\usepackage{amsmath, amssymb, amscd, amsthm, amsfonts}
\usepackage{mathdots}
\usepackage{mathrsfs}       

\usepackage{float}
\usepackage{graphicx}

\usepackage{bbm} 
\usepackage{microtype} 
\usepackage[normalem]{ulem}  

\usepackage[unicode, pdftex]{hyperref}
\graphicspath{{pictures}}
\usepackage{hyperref}

\usepackage{euscript}

\usepackage[usenames]{color}
\usepackage{colortbl}

\usepackage{colonequals} 
\usepackage{mathtools}

\usepackage[all]{xy}
\usepackage{caption}
\usepackage{subcaption}

\usepackage{tikz,tkz-euclide}
\usepackage{tikz-cd}
\usetikzlibrary{arrows}
\usetikzlibrary {patterns,patterns.meta}
\usepackage{tikz-3dplot}
\usetikzlibrary{shapes.geometric, calc} 

\usepackage{colonequals} 
\usepackage{blkarray, bigstrut} %
\usepackage{comment}

\def\le{\leqslant}
\def\ge{\geqslant}
\def\leq{\leqslant}
\def\geq{\geqslant}

\def\Z{\mathbb{Z}}
\def\R{\mathbb{R}}
\def\C{\mathbb{C}}

\def\B{\mathrm{Bier}}

\def\sK{\mathcal K}

\def\zk{\mathcal Z_{\mathcal K}}
\def\rk{\mathcal R_{\mathcal K}}

\def\zp{\mathcal{Z}_{P}}
\def\rp{\mathcal{R}_{P}}

\def\MF{\mathrm{MF}}

\newcommand{\hr}[2][]{\hyperref[#2]{#1~\ref{#2}}}

\makeatletter
\newtheorem*{rep@theorem}{\rep@title}
\newcommand{\newreptheorem}[2]{%
\newenvironment{rep#1}[1]{%
 \def\rep@title{#2 \ref{##1}}%
 \begin{rep@theorem}}%
 {\end{rep@theorem}}}
\makeatother

\newreptheorem{theorem}{Theorem}

\newtheorem{theorem}{Theorem}[section]
\newtheorem{proposition}[theorem]{Proposition}
\newtheorem{lemma}[theorem]{Lemma}
\newtheorem{prob}[theorem]{Problem}
\newtheorem{corollary}[theorem]{Corollary}

\theoremstyle{definition}
\newtheorem{example}[theorem]{Example}

\newtheorem{definition}[theorem]{Definition}

\numberwithin{equation}{section}

\DeclareMathAlphabet{\mathbbmsl}{U}{bbm}{m}{sl}

\makeatletter
\@namedef{subjclassname@2020}{%
  \textup{2020} Mathematics Subject Classification}
\makeatother

\begin{document}

\title[Bier spheres and Toric Topology]
{Bier spheres and Toric Topology}

\author[Limonchenko]{Ivan Limonchenko}
\address{Mathematical Institute of the Serbian Academy of Sciences and Arts (SASA), Belgrade, Serbia}
\email{ivan.limoncenko@turing.mi.sanu.ac.rs}

\author[Sergeev]{Matvey Sergeev}
\address{National Research University Higher School of Economics, Russian Federation} \email{matveys.studios@gmail.com}

\subjclass[2020]{13F55, 55N10, 57S12, 52B10, 52B05}

\keywords{Bier sphere, nestohedron, Delzant polytope, quasitoric manifold, small cover, Buchstaber number}

\begin{abstract}
We compute the real and complex Buchstaber numbers of an arbitrary Bier sphere.
In dimension two, we identify all the 13 different combinatorial types of Bier spheres and show
that 12 of them are nerve complexes of nestohedra, while the remaining one is a nerve complex
of a generalized permutohedron. As an application of our results, we construct a regular normal
fan for each of those 13 Delzant polytopes, compute the cohomology rings of the corresponding
nonsingular projective toric varieties, and examine the orientability of the corresponding small
covers.
\end{abstract}

\dedicatory{Dedicated to Professor Victor M. Buchstaber on the occasion of his 80th birthday}

\maketitle 

\section{Introduction}\label{section-introduction}
Since 1992, when Thomas Bier introduced a construction of a PL-sphere $\B(\sK)$ arising from an arbitrary simplicial complex $\sK$ on $[m]:=\{1,\ldots,m\}$ different from the whole simplex $\Delta_{[m]}$, see~\cite{Bier}, these triangulated spheres and their generalizations have been studied mainly in the context of topological combinatorics and combinatorial commutative algebra. The aim of this paper is to relate the geometrical and combinatorial properties of Bier spheres to the topological properties of their toric spaces arising in the framework of toric topology. 

By definition, $\B(\sK)$ is a deleted join of $\sK$ and its Alexander dual $\widehat{\sK}$; that is 
$$
\B(\sK):=\mathcal{K}*_{\Delta}\widehat{\mathcal{K}}=\{ \sigma \uplus \tau\text{ $|$ }\sigma \in \mathcal{K},\text{ }\tau\in \widehat{\mathcal{K}},\text{ }\sigma \cap \tau = \varnothing\},
$$
where $\sigma \uplus \tau := (\sigma \times \{1\})\sqcup (\tau\times \{2\})$. Apart from the one in the original note~\cite{Bier}, two different proofs of the crucial fact that for any simplicial complex $\sK\neq\Delta_{[m]}$ on $[m]$ this construction always yields an $(m-2)$-dimensional PL-sphere with a number of geometrical ('real') vertices varying between $m$ and $2m$ can be found in~\cite{Matousek} and~\cite{Longueville}. 

It follows directly from the above construction that the number of Bier spheres of a fixed dimension is finite. Moreover, it was shown in~\cite{BPSZ05} that most of the Bier spheres are non-polytopal, although the problem of constructing a particular example of a non-polytopal Bier sphere remains to be opened.  

Some of the known applications of Bier spheres in topological combinatorics are related to the fact that they have become a useful tool there thanks to their connection of their theory with the Van Kampen--Flores theorem, see~\cite{Matousek}. In addition, it was found in~\cite{ Zivaljevic19} that any Bier sphere admits a starshaped
realization. Therefore, each Bier sphere corresponds to at least one complete toric variety (possibly a singular one).

Furthermore, generalizations of Bier spheres to posets (\cite{BPSZ05}) and multicomplexes (\cite{Murai}) were proven to have nice combinatorial properties such as being shellable and edge decomposable.

Some of the known applications of Bier spheres in combinatorial commutative algebra are due to the results of~\cite{Murai} that established a relation of Bier spheres with polarizations and Alexander duality for monomial ideals.

In the framework of polytope theory, Bier spheres have been studied recently in~\cite{Zivaljevic19} and~\cite{Zivaljevic21}, where an effective criterion for their polytopality was obtained. Recently, in~\cite{Zivaljevic23}, the classical ``Steinitz problem'' in the case of Bier spheres was put into the context of game theory and related to the problem of characterizing whether a game is polytopal in the sense that the corresponding Bier sphere is polytopal.

The key result of~\cite{Zivaljevic21} relevant to our approach states that any polytopal Bier sphere is a nerve complex of a generalized permutohedron; that is, any polytopal Bier sphere is the boundary of a simplicial polytope dual to a generalized permutohedron.  

The class of generalized permutohedra was introduced in~\cite{Postnikov05} and the construction goes as follows. Let $S_{n}$ be the symmetric group on $n$ elements and fix a point $(x_{1},\ldots,x_{n})\in \mathbb{R}^{n}$. Then the corresponding classical permutohedron $P_{n}(x_{1},\ldots,x_{n})$ is defined to be a convex hull of the set $\{ x_\sigma:=(x_{\sigma(1)},\ldots,x_{\sigma(n)})\text{ }\mathrm{|} \text{ }\sigma \in S_{n} \}$. 

By definition, a generalized permutohedron is a convex hull of the set $\{ x_{\sigma}\in \mathbb{R}^{n}\text{ }\mathrm{|} \text{ }\sigma \in S_{n} \}$ such that for any adjacent transposition $\tau = (i, i+1)$ and any $\sigma \in S_{n}$, we have $x_{\sigma}-x_{\sigma \tau} = k_{\sigma, i}(e_{\sigma(i)}-e_{\sigma(i+1)})$, for $k_{\sigma,i}\geqslant 0$, where $e_{1},\ldots,e_{n}\in \mathbb{R}^{n}$ are the standard basis elements in $\mathbb{R}^{n}$. 

Each generalized permutohedron can be obtained by parallel translation of the facets of a certain classical permutohedron defined above (\cite{Postnikov05}). A particularly important subclass of generalized permutohedra is formed by all nestohedra. By definition, a nestohedron $P_{\mathcal{B}}$ is a Minkowski sum of regular simplices $\mathrm{conv}(e_i\,|\,i\in S)$, where $S$ varies over all proper elements of a building set $\mathcal{B}\subseteq 2^{[n]}$, see~\cite{Postnikov05, Postnikov-Reiner-Williams06}. 

In~\cite{Feichtner-Sturmfels04}, nerve complexes of nestohedra were introduced as nested set complexes and their combinatorial structure was studied. It turned out that each nestohedron is a Delzant polytope, which can be obtained from a simplex by cutting off some of its faces by hyperplanes in general position. In~\cite{Buchstaber-Volodin}, the case of flag nestohedra was studied and it was shown that each of them is a 2-truncated cube.

The structure of this paper is as follows. 

In Section 2 we recall all the necessary definitions related to simplicial complexes and simple polytopes. Then we show that there exist only 13 different combinatorial types of Bier spheres in dimension two. 
Since any 2-dimensional triangulated sphere is polytopal, we are able to sketch them as boundaries of the 3-dimensional simplicial polytopes and find geometric realizations of their dual simple polytopes as generalized truncation polytopes and nestohedra, whenever possible. In particular, we show that all the 13 simple polytopes arising in this way have Delzant realizations, all but one of them are nestohedra, and only 4 of them are flag polytopes (Theorem~\ref{MainClassTheo}).  

In Section 3 we briefly recall the notions of (real) moment-angle-complex and (real) Buchstaber number of the corresponding simplicial complex. These are the key notions of toric topology, for details we refer to the fundamental monograph~\cite{Buchstaber-Panov}. We construct a characteristic map for any Bier sphere, not necessarily a polytopal one, and use it to show that for any simplicial complex $\sK\neq\Delta_{[m]}$ on $[m]$, the real and complex Buchstaber numbers of its Bier sphere are both equal to $m+1$, for any $m\geq 2$ (Theorem~\ref{BuchNumTheo}).

In Section 4 we consider some examples of our main results application. We briefly recall the notions of a quasitoric manifold and a small cover and construct the canonical regular normal fans for each of the 12 nestohedra provided by our Bier spheres classification using the result of~\cite{Fenn}. Then we compute cohomology rings of the resulting 13 nonsingular projective toric varieties (Example~\ref{QuasitoricEx}) and identify orientable manifolds among the resulting 13 small covers (Example~\ref{SmallCoverEx}).

There are two appendices in the paper. In Appendix A we collect all the combinatorial information (the sets of minimal non-faces and the $f$-vectors) of the two-dimensional Bier spheres, which is necessary for their classification. In Appendix B we collect the characteristic matrices, or in geometrical terms, the regular normal fan generators, for each of the 13 simple $3$-polytopes corresponding to the combinatorial types of two-dimensional Bier spheres.


\section{Two-dimensional case: classification}\label{section 2}

In this section we are going to classify the combinatorial types of all two-dimensional Bier spheres and study their combinatorial properties.

\subsection{Preliminaries}

We start by recalling the main definitions concerned with simplicial complexes and simple polytopes that are used in what follows.

\begin{definition}
Let $V$ be a finite set. We say that $\sK\subseteq 2^V$ is a {\emph{simplicial complex on}} $V$ if the following two conditions hold:
\begin{enumerate}
\item $\varnothing\in \sK$; and
\item $\sigma\in \sK, \tau\subseteq\sigma\Rightarrow \tau\in \sK$.
\end{enumerate}
Elements of $\sK$ are called its {\emph{faces}} or {\emph{simplices}}. Faces of cardinality $1$ are called {\emph{vertices}}, faces of cardinality $2$ are called {\emph{edges}}. 

Maximal (by inclusion) simplices of $\sK$ are called its {\emph{facets}} and their set is denoted by $M(\sK)$. Minimal (by inclusion) subsets in $V$ that are not simplices of $\sK$ are called its {\emph{minimal non-faces}} and their set is denoted by $\MF(\sK)$. Minimal non-faces of cardinality $1$ are called {\emph{ghost vertices}}.

The {\emph{dimension}} of $\sK$ is one less than the maximal cardinality of a face of $\sK$ and it is denoted by $\dim \sK$. If all maximal faces of $\sK$ have the same cardinality, then $\sK$ is called {\emph{pure}}. 

Suppose $\sK$ is an $(n-1)$-dimensional simplicial complex. By the {\emph{$f$-vector}} $f(\sK)$ of $\sK$ we mean a non-negative integer vector $(f_0,\ldots,f_{n-1})$, where $f_i$ equals the number of $i$-dimensional faces of $\sK$ for $0\leq i\leq n-1$. Note that one can also set $f_{-1}=1$ for any simplicial complex. By the {\emph{$h$-vector}} $h(\sK)$ of $\sK$ we mean an integer vector $(h_0,\ldots,h_n)$ determined by the formula:
$$
h_0 t^n+\ldots+h_{n-1}t+h_n=(t-1)^n+f_0 (t-1)^{n-1}+\ldots+f_{n-1}.
$$

We say that two simplicial complexes, $\sK_1$ and $\sK_2$, are {\emph{combinatorially equivalent}}, or {\emph{isomorphic}}, if there exists a bijection between their vertex sets such that it maps simplices to simplices in both directions.
\end{definition}

Thus, $\varnothing_V$ and $\Delta_V$, the empty set singleton simplicial complex and the entire power set of $V$, are simplicial complexes on $V$ of dimensions $-1$ and $|V|-1$, respectively. Moreover, both are pure and $\MF(\Delta_V)=\varnothing$, while $\varnothing_V$ has $|V|$ ghost vertices. Another example of a pure complex is given by any {\emph{triangulated sphere}}; that is, a $q$-dimensional simplicial complex homeomorphic to the $q$-sphere.  

\begin{definition}
A simplicial complex $\sK$ is called {\emph{flag}} if one of the following equivalent conditions holds:
\begin{itemize}
\item elements of $\MF(\sK)$ have cardinalities $\leq 2$;
\item each set of vertices of $\sK$ pairwisely linked by edges in $\sK$ is itself in $\sK$.
\end{itemize}    
\end{definition}

\begin{definition}
Let $[m]:=\{1,2,\ldots,m\}$ and $[m']:=\{1',2',\ldots,m'\}$ be two ordered sets and let the map $\phi\colon i\mapsto i', 1\leq i\leq m$ be an order preserving bijection between them.

Suppose $\sK$ is a simplicial complex on $[m]$ and $\sK\neq\Delta_{[m]}$. Then 
\begin{itemize}
\item its {\emph{Alexander dual}} is a simplicial complex $\widehat{\sK}$ on $[m']$ such that 
$$
I\in \MF(\sK)\Longleftrightarrow [m']\setminus\phi(I)\in M(\widehat{\sK}).
$$
\item its {\emph{Bier sphere}} is a simplicial complex $\B(\sK)$ on $[m]\sqcup [m']$ such that 
$$
\B(\sK):=\{I\sqcup \phi(J)\,|\,I\in \sK, \phi(J)\in \widehat{\sK}, I\cap J=\varnothing\}=:\mathcal{K}*_{\Delta}\widehat{\mathcal{K}};
$$
that is, $\B(\sK)$ is defined to be the {\emph{deleted join}} of $\sK$ and $\widehat{\sK}$.
\end{itemize}
\end{definition}

The latter construction was introduced in the unpublished note ~\cite{Bier}, where it was also proved that $\B(\sK)$ is an $(m-2)$-dimensional {\emph{PL-sphere}} with a number of vertices varying between $m$ and $2m$, for any simplicial complex $\sK\neq\Delta_{[m]}$ on $[m]$. This means that there exists a subdivision of the $(m-2)$-dimensional triangulated sphere $\B(\sK)$ combinatorially equivalent to a subdivision of $\partial\Delta^{m-1}$, the boundary of the $(m-1)$-dimensional simplex. 
 
\begin{example}
Observe that, by the above definition, we have: 
$$
\B(\varnothing_{[m]})=\widehat{\varnothing_{[m]}}=\partial\Delta_{[m']}\text{ and }\B(\partial\Delta_{[m]})=\partial\Delta_{[m]},
$$
both being $(m-2)$-dimensional PL-spheres on $[m]\sqcup [m']$, with the set of vertices $[m']$ and the set of ghost vertices $[m]$ and vice versa, respectively.
\end{example}

\begin{definition}
By an $n$-dimensional {\emph{simplicial polytope}} $S$ we mean a convex hull of a set of points in $\R^n$ that are in general position. It is combinatorially dual to an $n$-dimensional {\emph{simple polytope}} $P$; that is, the supporting hyperplanes of $P$ are in general position (or equivalently, each vertex of $P$ is an intersection of precisely $n$ facets of $P$) and there exists an inclusion reversing bijection between the face lattices of $S$ and $P$.
\end{definition}

Furthermore, for any $n$-dimensional simple polytope $P$, the boundary set
$$
K_P:=\partial S=\partial P^*
$$
is a pure $(n-1)$-dimensional simplicial complex, which is known to be a PL-sphere. We call it the {\emph{nerve complex}} of $P$, since it is isomorphic to the nerve of the closed covering of $\partial P$ by its facets. Given a simple polytope $P$, we define the {\emph{$f$-vector}} $f(P)$ of $P$ to be $f(K_P)$ and the {\emph{$h$-vector}} $h(P)$ of $P$ to be $h(K_P)$.

In this paper, we are only interested in {\emph{combinatorial types}} of simplicial complexes and simple polytopes arising from Bier spheres; that is, in equivalence classes with respect to the equivalence relation of being combinatorially equivalent. Note that two simple polytopes, $P_1$ and $P_2$, are {\emph{combinatorially equivalent}}, that is, their face lattices are isomorphic, if and only if their nerve complexes, $K_{P_1}$ and $K_{P_2}$, are isomorphic as simplicial complexes. Furthermore, for any simple polytope $P$, the two integer vectors, $f(P)$ and $h(P)$, are combinatorial invariants of $P$.

Now we are ready to turn to the classification of combinatorial types of the two-dimensional Bier spheres.

\subsection{Combinatorial types of \texorpdfstring{$\B_4(\sK)$}{Lg}}

The classification will be performed by sorting through all possible simplicial complexes on the set $[4]=\{1,2,3,4\}$. As is well-known, all 2-dimensional triangulated spheres are polytopal, hence we can talk about simplicial Bier spheres $\partial S=\partial P^*$ and the corresponding simple {\emph{Bier polytopes}} $P$ interchangeably in this context.

Recall that we allow simplicial complexes to have ghost vertices. By virtue of the equality $\mathrm{Bier}(\mathcal{K})=\mathrm{Bier}(\widehat{\mathcal{K}})$, it is enough for our purposes to check simplicial complexes up to combinatorial equivalence and Alexander duality. 

The combinatorial types of two-dimensional Bier spheres are identified based on the following three simple observations:
\begin{enumerate}
    \item We can distinguish them by means of $f$- and $h$-vectors, being combinatorial invariants of simplicial complexes. Moreover, a flag polytope cannot be isomorphic to a non-flag polytope, so we compare their $f$- and $h$-vectors separately;
    \item For polytopes with the same $f$- or $h$-vectors, we compare the numbers of elements in the sets of minimal non-faces of their nerve complexes;
    \item If all simple $3$-cycles of two simplicial $3$-polytopes are the boundaries of their faces, then the isomorphism of their graphs entails their simplicial isomorphism.
\end{enumerate}

Denote by $x_{i}$ labeled vertices of $\mathcal{K}$, by $y_{i}$ -- labeled vertices of $\widehat{\mathcal{K}}$. The following well-known result allows us to describe the set of minimal non-faces of a Bier sphere; we include a proof for the sake of completeness.

\begin{proposition}\hypertarget{propMF}
Let $\{x_{1},\ldots,x_{m}\}$ and $\{y_{1},\ldots,y_{m}\}$ be sets of labeled vertices of simplicial complexes $\mathcal{K}$ and $\widehat{\mathcal{K}}$ respectively, and the vertices $x_{i}$, $y_{i}$ correspond to the same vertex $\{i\}$. Then we have:
$$
\mathrm{MF}(\mathrm{Bier}(\mathcal{K}))=\mathrm{MF}(\mathcal{K})\sqcup \mathrm{MF}(\widehat{\mathcal{K}})\sqcup \{ x_{i_{1}}y_{i_{1}}, \ldots, x_{i_{n}}y_{i_{n}}\},
$$ 
for some $n\le m$, where $x_{i_{k}}y_{i_{k}}$ for  $k=1,\ldots,n$ are all the pairs of vertices, for which both $x_{i_{k}},y_{i_{k}}$ are not ghost.
\end{proposition}
\begin{proof}
Let us describe the minimal non-faces of $\mathcal{K}*_{\Delta}\widehat{\mathcal{K}}$. These are $\sigma \uplus \tau$ such that either $\sigma \uplus \tau \notin \mathcal{K}*\widehat{\mathcal{K}}$, or $\sigma\uplus\tau\in \mathcal{K}*\widehat{\mathcal{K}}$ and $\sigma\cap\tau\ne\varnothing$.

If $\sigma \uplus \tau \in \mathcal{K}*\widehat{\mathcal{K}}$ and $\sigma\cap\tau\ne\varnothing$, then the minimality of this non-face means that any proper face of $\sigma$ does not intersect any (not necessarily proper) face of $\tau$ and vice versa. This is possible if and only if $\sigma = \tau = \{i\}$ for a vertex $\{i\}$ that belongs to $\mathcal{K}$ and $\widehat{\mathcal{K}}$. Thus, we get $\{x_{i},y_{i}\} = x_{i}y_{i} \in \mathrm{MF}(\mathrm{Bier}(\mathcal{K}))$. Note that if $\{i\}$ is a ghost vertex for $\mathcal{K}$ or $\widehat{\mathcal{K}}$, then $x_{i}y_{i}\notin \mathrm{MF}(\mathrm{Bier}(\mathcal{K}))$.

If $\sigma \uplus \tau \notin \mathcal{K}*\widehat{\mathcal{K}}$, then we get either $\sigma\notin \mathcal{K}$ and $\tau\in \widehat{\mathcal{K}}$, or $\sigma \in \mathcal{K}$ and $\tau\notin \widehat{\mathcal{K}}$. In the former case, the minimality means $\sigma\in\mathrm{MF}(\mathcal{K})$, $\tau = \varnothing$, and any proper face of $\sigma$ does not intersect $\tau$ automatically.  In the latter case, the minimality similarly means that $\tau\in\mathrm{MF}(\widehat{\mathcal{K}})$ and $\sigma=\varnothing$.
\end{proof}

\begin{corollary}
Let $\sK$ be an arbitrary simplicial complex on $V$. Then the number of ghost vertices of $\B(\mathcal{K})$ is equal to $|V|-f_0(\sK)+f_{|V|-2}(\sK)$.
\end{corollary}
\begin{proof}
Indeed, the ghost vertices of the Bier sphere $\B(\sK)$ are exactly the singleton elements
of $\MF(\B(\mathcal{K}))$. Those, in turn, are the singleton elements of either $\mathrm{MF}(\mathcal{K})$ or $\mathrm{MF}(\widehat{\mathcal{K}})$, by the
previous result. By definition, the simplices of $\widehat{\mathcal{K}}$ are the complements of non-simplices of $\mathcal{K}$. Thus,
the ghost vertices of $\widehat{\mathcal{K}}$ are exactly the complements of the faces of $\mathcal{K}$ of cardinality $|V|-1$.
\end{proof}

Observe that there are exactly $28$ simplicial complexes on $4$ vertices (without counting the full simplex itself).
\begin{center}
    \input{name}
\end{center}
The Bier polytopes corresponding to $\mathrm{Bier}_{4}(\mathcal{K}_{11})$ and $\mathrm{Bier}_{4}(\mathcal{K}_{8})$ are both combinatorially equivalent to the 3-cube with one vertex cut. 
Thus, these two spheres are isomorphic. 

Next, it is easy to see that every $3$-cycle of $\mathrm{Bier}_{4}(\mathcal{K}_{6})$ and $\mathrm{Bier}_{4}(\mathcal{K}_{4})$ is a boundary of some face. By checking the isomorphism of their graphs, we obtain the isomorphism 
$$
\mathrm{Bier}_{4}(\mathcal{K}_{6}) \to \mathrm{Bier}_{4}(\mathcal{K}_{4}),\text{ } x_{1}\mapsto y_{4},x_{2}\mapsto x_{3},x_{3}\mapsto x_{4},x_{4}\mapsto y_{3},y_{1}\mapsto y_{2},y_{2}\mapsto x_{2},y_{3}\mapsto x_{1},y_{4}\mapsto y_{1}.
$$ 

Furthermore, $\mathrm{Bier}_{4}(\mathcal{K}_{9})$ and $\mathrm{Bier}_{4}(\mathcal{K}_{13})$ are also isomorphic: 
$$
\mathrm{Bier}_{4}(\mathcal{K}_{9})\to \mathrm{Bier}_{4}(\mathcal{K}_{13}),\text{ }x_{1}\mapsto x_{2},\text{ }x_{2}\mapsto x_{3},\text{ }x_{3}\mapsto y_{1},\text{ }y_{1}\mapsto y_{2},\text{ }y_{2}\mapsto x_{4},\text{ }y_{3}\mapsto x_{1}.
$$
Starting from now, we denote
\begin{center}
$$
\EuScript{S}_{1}:=\mathrm{Bier}_{4}(\mathcal{K}_{1}),\quad
\EuScript{S}_{2}:=\mathrm{Bier}_{4}(\mathcal{K}_{2}),\quad \EuScript{S}_{3}:=\mathrm{Bier}_{4}(\mathcal{K}_{3}),$$
$$
\EuScript{S}_{4}:=\mathrm{Bier}_{4}(\mathcal{K}_{4})\cong \mathrm{Bier}_{4}(\mathcal{K}_{6}),\quad 
\EuScript{S}_{5}:=\mathrm{Bier}_{4}(\mathcal{K}_{5}),\quad \EuScript{S}_{6}:=\mathrm{Bier}_{4}(\mathcal{K}_{7}),$$
$$
\EuScript{S}_{7}:=\mathrm{Bier}_{4}(\mathcal{K}_{8})\cong \mathrm{Bier}_{4}(\mathcal{K}_{11}),\quad 
\EuScript{S}_{8}:=\mathrm{Bier}_{4}(\mathcal{K}_{10}),\quad \EuScript{S}_{9}:=\mathrm{Bier}_{4}(\mathcal{K}_{12}),$$
$$
\EuScript{S}_{10}:=\mathrm{Bier}_{4}(\mathcal{K}_{13})\cong \mathrm{Bier}_{4}(\mathcal{K}_{9}),\quad 
\EuScript{S}_{11}:=\mathrm{Bier}_{4}(\mathcal{K}_{14}),\quad 
\EuScript{S}_{12}:=\mathrm{Bier}_{4}(\mathcal{K}_{15}),\quad 
\EuScript{S}_{13}:=\mathrm{Bier}_{4}(\mathcal{K}_{16}).
$$
\end{center}
Finally, set $\EuScript{P}_{i}$ for the Bier polytope corresponding to $\EuScript{S}_{i}$; that is, $\EuScript{P}_{i}$ is the simple 3-polytope having $\EuScript{S}_{i}$ as its nerve complex, for each $1\leq i\leq 13$.

To complete our combinatorial type classification we shall show that the 13 triangulated spheres above are pairwisely not combinatorially equivalent. In order to do this, we shall make use of the combinatorial data we collected in \hyperlink{Appendix A}{Appendix A}. 

Firstly, note that the simple polytope $\EuScript{P}_{4}$ has pentagonal faces, while $\EuScript{P}_{5}$ does not. Therefore the Bier sphere $\EuScript{S}_{4}$ is not combinatorially equivalent to the Bier sphere $\EuScript{S}_{5}$. 

Secondly, observe that the triangulated spheres $\EuScript{S}_{1},\EuScript{S}_{2},\EuScript{S}_{3}$ can be distinguished pairwisely by the numbers of elements in their sets of minimal non-faces. The same is true for the Bier spheres $\EuScript{S}_{7}$ and $\EuScript{S}_{9}$.

Finally, note that the triangulated sphere $\EuScript{S}_{6}$ is not isomorphic to the triangulated sphere $\EuScript{S}_{2}$, since $\EuScript{P}_{2}$ has a hexagonal face, while $\EuScript{P}_{6}$ does not. 

This finishes the classification of combinatorial types of all the two-dimensional Bier spheres. The thirteen resulting Bier polytopes and their duals are realized as convex polytopes in the 3-space in the next subsection. 

\subsection{Affine realizations and generalized truncation polytopes}

Here we sketch the affine realizations of $\EuScript{S}_{i}$ and $\EuScript{P}_{i}$ in $\mathbb{R}^{3}$, for all $1\leq i\leq 13$, see \hyperlink{Table 2}{Table 2}.

\hypertarget{Table 2}{
\input{union_table}
}

The following class of simple polytopes was introduced and studied in the framework of toric topology in~\cite{Limonchenko14}.

\begin{definition}
By a {\emph{generalized truncation polytope}} of the type $vc^k(\Delta^{n_1}\times\ldots\times\Delta^{n_r})$ for $k\geq 0$ and $n_1\geq\ldots\geq n_r\geq 1$ we mean (a representative of) the combinatorial type of a simple polytope $P$ being the result of a consecutive cut of $k$ vertices by hyperplanes in general position, starting from the product of simplices $\Delta^{n_1}\times\ldots\times\Delta^{n_r}$.

When $r=1$, we call $P$ a {\emph{truncation polytope}} and its dual $S=P^*$ -- a {\emph{stacked polytope}}.
\end{definition}

Note that all simple 3-dimensional polytopes are face truncations of the 3-simplex $\Delta^{3}$. A particularly important is the following class of flag polytopes, introduced in~\cite{Buchstaber-Volodin}. 

\begin{definition}
We say that a simple polytope $P$ is an $n$-dimensional {\emph{2-truncated cube}}, if it can be obtained from the $n$-cube $I^n$ by consecutively cutting off faces of codimension-2 by hyperplanes in general position.    
\end{definition}

Observe that a face truncation of a cube is flag if and only if it is a 2-truncated cube.

Now we return back to our classification of Bier spheres. 

One immediately sees that: $\EuScript{P}_{1}$ is $vc^{4}(\Delta^{3})$, $\EuScript{P}_{2}$ is $vc^{2}(\Delta^1\times\Delta^1\times\Delta^1)$ (cut off two adjacent vertices of the $3$-cube), $\EuScript{P}_{3}$ is a vertex cut from a $5$-angle prism (took the $3$-cube and cut off its edge first, then the vertex), $\EuScript{P}_{4}$ is $2$-truncated cube (cut off two adjacent edges of the $3$-cube; as a result, the pentagonal face $x_{4}x_{2}x_{6}x_{10}x_{9}$ is formed), $\EuScript{P}_{5}$ is $2$-truncated cube (cut off two parallel edges from one facet of the $3$-cube), $\EuScript{P}_{6}$ is $vc^{2}(\Delta^{1}\times\Delta^{1}\times\Delta^{1})$ (cut off two opposite vertices of the $3$-cube). 

Furthermore, $\EuScript{P}_{7}$ is $vc^{1}(\Delta^{1}\times\Delta^{1}\times\Delta^{1})$, $\EuScript{P}_{8}$ is $2$-truncated cube (with one $2$-truncation performed), $\EuScript{P}_{9}$ is $vc^{3}(\Delta^{3})$, $\EuScript{P}_{10}=\Delta^{1}\times\Delta^{1}\times\Delta^{1}$, $\EuScript{P}_{11}$ is $vc^{2}(\Delta^{3})$, $\EuScript{P}_{12}$ is $vc^{1}(\Delta^{3})$, $\EuScript{P}_{13}=\Delta^{3}$.

\subsection{Delzant realizations and nestohedra}

We start with the following definition.

\begin{definition}
A \emph{Delzant polytope} is a lattice polytope such that its normal fan is regular; i.e., the primitive normal vectors to facets meeting at a vertex form a lattice basis for each vertex of the lattice polytope. A \emph{Delzant realization} of a simple polytope $P$ is an arbitrary Delzant polytope which is combinatorially equivalent to $P$.
\end{definition}

It was shown in~\cite{Del05} that if a 3-polytope has a Delzant realization, then it has at least
one triangular or quadrangular facet. Two classes of Delzant polytopes are extremely important for our
classification: nestohedra and 2-truncated cubes. 

Let $S_{n}$ be the symmetric group on $n$ elements and $(x_{1},\ldots,x_{n})\in \mathbb{R}^{n}$. Recall that the corresponding {\emph{classical permutohedron}} $P_{n}(x_{1},\ldots,x_{n})$ is a convex hull of the set $\{x_{\sigma}:=(x_{\sigma(1)},\ldots,x_{\sigma(n)})\text{ }\mathrm{|} \text{ }\sigma \in S_{n} \}$. 

By definition, a {\emph{generalized permutohedron}} is a convex hull of the set 
$$
\{ x_{\sigma}\in \mathbb{R}^{n}\text{ }\mathrm{|} \text{ }\sigma \in S_{n} \}
$$ 
such that for any adjacent transposition $\tau = (i, i+1)$ and any $\sigma \in S_{n}$, we have 
$$
x_{\sigma}-x_{\sigma \tau} = k_{\sigma, i}(e_{\sigma(i)}-e_{\sigma(i+1)}),\text{ for } k_{\sigma,i}\geqslant 0,
$$
where $e_{1},\ldots,e_{n}\in \mathbb{R}^{n}$ are the standard basis elements in $\mathbb{R}^{n}$. 

It follows that each generalized permutohedron can be obtained by parallel translation of the facets of a certain classical permutohedron. A particularly important subclass of generalized permutohedra is formed by all nestohedra. 

By definition, a {\emph{nestohedron}} $P_{\mathcal{B}}$ is a Minkowski sum of regular simplices $\mathrm{conv}(e_i\,|\,i\in S)$, where $S$ varies over all proper elements of a building set $\mathcal{B}\subseteq 2^{[n+1]}$, see~\cite{Postnikov05, Postnikov-Reiner-Williams06}. A {\emph{building set}} satisfies the following two conditions:
\begin{enumerate}
\item If $S_{1},S_{2}\in \mathcal{B}$ such that $S_{1}\cap S_{2}\ne \varnothing$, then $S_{1}\cup S_{2}\in \mathcal{B}$;
\item $\{i\}\in \mathcal{B}$ for all $i\in [n+1]$.
\end{enumerate}
By definition, a nestohedron $P_{\mathcal{B}}$ is a lattice polytope with respect to the lattice $$\{(x_{1},\dots, x_{n+1})\in \mathbb{Z}^{n+1}\text{ | } \sum x_i = |\mathcal{B}| \}.$$

A building set $\mathcal{B}\subset 2^{[n+1]}$ is called \emph{connected} if $[n+1]\in \mathcal{B}$. It follows from the Erokhovets construction (see. ~\cite[Construction 4.1]{Buchstaber-Volodin}) that each nestohedron can be realized with a connected building set, see~\cite[Corollary 4.5]{Buchstaber-Volodin}. Therefore, without loss of generality, we can restrict ourselves to nestohedra defined on connected building sets.

The geometric description of the faces of a nestohedron is well-known ~\cite[Proposition 7.5]{Postnikov05}
\begin{proposition}[\cite{Postnikov05}]
Let $\mathcal{B}$ be a connected building set on $[n+1]$. Then
nestohedron $P_{\mathcal{B}}$ is an $n$-dimensional simple polytope
equal to the intersection of the hyperplane  $H=\{(x_{1},\dots,
x_{n+1})\in \mathbb{R}^{n+1}\emph{\text{ | }}
\sum^{n+1}_{i=1}x_{i}=|\mathcal{B}|\}$ with the half space
$H_{S}=\{(x_{1},\dots, x_{n+1})\in \mathbb{R}^{n+1}\emph{\text{ |
}}\sum_{i\in S}x_{i} \ge |\mathcal{B}|_{S}| \}$, where
$\mathcal{B}|_{S}=\{S'\in \mathcal{B}\emph{\text{ | }}S'\subset S\}$.
Moreover, there is a one to one correspondence between facets of
$P_{\mathcal{B}}$ and elements of $\mathcal{B}\backslash \{[n+1]\}$,
such that $S\in \mathcal{B}$ corresponds to the facet
$F_{S}=P_{\mathcal{B}}\cap \partial H_{S}$.
\end{proposition}

In~\cite{Feichtner-Sturmfels04}, an alternative approach was implemented, leading to the description of the nerve complexes of nestohedra, which appeared there as {\emph{nested set complexes}}.

The following lemma is crucial for us~\cite[Theorem 4.2]{Feichtner-Sturmfels04}. We follow the notations of~\cite[Lemma 6.4]{Buchstaber-Volodin}.
\begin{lemma}[\cite{Feichtner-Sturmfels04, Buchstaber-Volodin, Buchstaber-Panov}]\label{lemma}
 Let $\mathcal{B}_{0}\subset \mathcal{B}_{1}$ be connected building sets on $[n+1]$. Then $P_{\mathcal{B}_{1}}$ is obtained from $P_{\mathcal{B}_{0}}$ by a sequence of truncations of the faces $G_{i}=\bigcap^{k_{i}}_{j=1}F_{S_{j}^{i}}$ corresponding to the decompositions $S^{i}=S^{i}_{1}\sqcup...\sqcup S^{i}_{k_{i}}$ of elements $S^{i}\in \mathcal{B}_{1}\backslash \mathcal{B}_{0}$ numbered in
any order that is reverse to inclusion (i.e. $S^{i}\supseteq S^{i'}\Rightarrow i\le i'$).  
\end{lemma}

It follows from the last construction that any nestohedron can be obtained by face truncations of a simplex. Hence, every nestohedron is a Delzant polytope. So, every simple polytope $P$ is either combinatorially equivalent to some nestohedron (in this situation the nestohedron appears to be the “canonical” Delzant realization of $P$), or does not admit a realization as a nestohedron (then a Delzant realization might not exist).

The crucial result relating Bier spheres and nestohedra is the next one, see~\cite[Сorollary 19]{Zivaljevic21}.

\begin{theorem}[\cite{Zivaljevic21}]
Any polytopal Bier sphere is a nerve complex of a generalized permutohedron.
\end{theorem}

Now, we are going to show that all the 3-dimensional Bier polytopes are  nestohedra, except for one, which is merely a generalized permutohedron.

\begin{theorem}\label{NotDelzantTheorem1}
$\EuScript{P}_{6}$ is not a nestohedron.    
\end{theorem}
\begin{proof}
Suppose that $\EuScript{P}_{6}$ has a combinatorial type of nestohedron with a building set $\mathcal{B}$. Let $$\mathcal{B}_{\Delta^{3}} = \Big{\{}\{1\}, \{2\}, \{3\}, \{4\}, \{1,2,3,4\} \Big{\}}$$ be a building set of $\Delta^{3}$. It is well-known that every $n$-dimensional nestohedron can be obtained from the $n$-simplex{\footnote{In the flag case every $n$-dimensional nestohedron is an $n$-dimensional 2-truncated cube, see~\cite[Theorem 6.6]{Buchstaber-Volodin}}} by a sequence of face truncations, see~\cite[Theorem 1.5.18]{Buchstaber-Panov}. In other words, $\mathcal{B}_{\Delta^{3}}\subset \mathcal{B}$.

Note that cutting off the vertex of a simple $3$-polytope adds $1$ triangular face, $3$ edges and $2$ vertices, and slicing an edge adds $1$ quadrangular face, $3$ edges, and $2$ vertices. The $f$-vector of $\EuScript{P}_{6}$ is $(12,18,8)$ and the $f$-vector of $\Delta^3$ is $(4,6,4)$. Thus, it is necessary to do exactly $\frac{12-4}{2}=4$ cuts to get $\EuScript{P}_{6}$ from $\Delta^3$. Observe that $\EuScript{P}_{6}$ does not have a quadrangular face. So the last step should be the vertex truncation.

\emph{Case 1}. $\mathcal{B}$ contains a building set $\mathcal{B}_{I^{3}}$ of the $3$-cube $I^{3}$. Without loss of generality, let $$\mathcal{B}_{I^{3}}=\Big{\{}\{1\}, \{2\}, \{3\}, \{4\}, \{1,2\}, \{1,2,3\}, \{1,2,3,4\} \Big{\}}.$$
\hypertarget{Figure 1}{
\begin{figure}[ht]
\begin{center}
\usetikzlibrary{patterns}

\tikzset{
pattern size/.store in=\mcSize, 
pattern size = 5pt,
pattern thickness/.store in=\mcThickness, 
pattern thickness = 0.3pt,
pattern radius/.store in=\mcRadius, 
pattern radius = 1pt}
\makeatletter
\pgfutil@ifundefined{pgf@pattern@name@_4410961wf}{
\pgfdeclarepatternformonly[\mcThickness,\mcSize]{_4410961wf}
{\pgfqpoint{0pt}{0pt}}
{\pgfpoint{\mcSize+\mcThickness}{\mcSize+\mcThickness}}
{\pgfpoint{\mcSize}{\mcSize}}
{
\pgfsetcolor{\tikz@pattern@color}
\pgfsetlinewidth{\mcThickness}
\pgfpathmoveto{\pgfqpoint{0pt}{0pt}}
\pgfpathlineto{\pgfpoint{\mcSize+\mcThickness}{\mcSize+\mcThickness}}
\pgfusepath{stroke}
}}
\makeatother

 
\tikzset{
pattern size/.store in=\mcSize, 
pattern size = 5pt,
pattern thickness/.store in=\mcThickness, 
pattern thickness = 0.3pt,
pattern radius/.store in=\mcRadius, 
pattern radius = 1pt}
\makeatletter
\pgfutil@ifundefined{pgf@pattern@name@_ap7d8q7tq}{
\pgfdeclarepatternformonly[\mcThickness,\mcSize]{_ap7d8q7tq}
{\pgfqpoint{0pt}{0pt}}
{\pgfpoint{\mcSize+\mcThickness}{\mcSize+\mcThickness}}
{\pgfpoint{\mcSize}{\mcSize}}
{
\pgfsetcolor{\tikz@pattern@color}
\pgfsetlinewidth{\mcThickness}
\pgfpathmoveto{\pgfqpoint{0pt}{0pt}}
\pgfpathlineto{\pgfpoint{\mcSize+\mcThickness}{\mcSize+\mcThickness}}
\pgfusepath{stroke}
}}
\makeatother

 
\tikzset{
pattern size/.store in=\mcSize, 
pattern size = 5pt,
pattern thickness/.store in=\mcThickness, 
pattern thickness = 0.3pt,
pattern radius/.store in=\mcRadius, 
pattern radius = 1pt}
\makeatletter
\pgfutil@ifundefined{pgf@pattern@name@_m4ws0y35t}{
\pgfdeclarepatternformonly[\mcThickness,\mcSize]{_m4ws0y35t}
{\pgfqpoint{0pt}{-\mcThickness}}
{\pgfpoint{\mcSize}{\mcSize}}
{\pgfpoint{\mcSize}{\mcSize}}
{
\pgfsetcolor{\tikz@pattern@color}
\pgfsetlinewidth{\mcThickness}
\pgfpathmoveto{\pgfqpoint{0pt}{\mcSize}}
\pgfpathlineto{\pgfpoint{\mcSize+\mcThickness}{-\mcThickness}}
\pgfusepath{stroke}
}}
\makeatother

 
\tikzset{
pattern size/.store in=\mcSize, 
pattern size = 5pt,
pattern thickness/.store in=\mcThickness, 
pattern thickness = 0.3pt,
pattern radius/.store in=\mcRadius, 
pattern radius = 1pt}
\makeatletter
\pgfutil@ifundefined{pgf@pattern@name@_rvu75omry}{
\pgfdeclarepatternformonly[\mcThickness,\mcSize]{_rvu75omry}
{\pgfqpoint{0pt}{-\mcThickness}}
{\pgfpoint{\mcSize}{\mcSize}}
{\pgfpoint{\mcSize}{\mcSize}}
{
\pgfsetcolor{\tikz@pattern@color}
\pgfsetlinewidth{\mcThickness}
\pgfpathmoveto{\pgfqpoint{0pt}{\mcSize}}
\pgfpathlineto{\pgfpoint{\mcSize+\mcThickness}{-\mcThickness}}
\pgfusepath{stroke}
}}
\makeatother

 
\tikzset{
pattern size/.store in=\mcSize, 
pattern size = 5pt,
pattern thickness/.store in=\mcThickness, 
pattern thickness = 0.3pt,
pattern radius/.store in=\mcRadius, 
pattern radius = 1pt}
\makeatletter
\pgfutil@ifundefined{pgf@pattern@name@_n8caparkv}{
\pgfdeclarepatternformonly[\mcThickness,\mcSize]{_n8caparkv}
{\pgfqpoint{0pt}{0pt}}
{\pgfpoint{\mcSize+\mcThickness}{\mcSize+\mcThickness}}
{\pgfpoint{\mcSize}{\mcSize}}
{
\pgfsetcolor{\tikz@pattern@color}
\pgfsetlinewidth{\mcThickness}
\pgfpathmoveto{\pgfqpoint{0pt}{0pt}}
\pgfpathlineto{\pgfpoint{\mcSize+\mcThickness}{\mcSize+\mcThickness}}
\pgfusepath{stroke}
}}
\makeatother
\tikzset{every picture/.style={line width=0.75pt}} 

\begin{tikzpicture}[x=0.75pt,y=0.75pt,yscale=-1,xscale=1]

\draw  [draw opacity=0][pattern=_4410961wf,pattern size=3.75pt,pattern thickness=0.75pt,pattern radius=0pt, pattern color={rgb, 255:red, 0; green, 0; blue, 0}] (192.18,65.72) -- (141.81,79.17) -- (119.68,64.76) -- cycle ;
\draw   (228.33,117.11) -- (127.58,144) -- (83.33,115.19) -- (156.03,14.33) -- cycle ;
\draw    (156.03,14.33) -- (127.58,144) ;
\draw  [dash pattern={on 0.84pt off 2.51pt}]  (83.33,115.19) -- (228.33,117.11) ;
\draw    (119.68,64.76) -- (141.81,79.17) ;
\draw    (141.81,79.17) -- (192.18,65.72) ;
\draw  [dash pattern={on 0.84pt off 2.51pt}]  (119.68,64.76) -- (192.18,65.72) ;

\draw   (403.33,38) -- (319.33,82) -- (262,37) -- cycle ;
\draw    (319.33,82) -- (319.33,139) ;
\draw    (403.33,38) -- (403.33,95) ;
\draw    (262,37) -- (262,94) ;
\draw    (262,94) -- (319.33,139) -- (403.33,95) ;
\draw  [dash pattern={on 0.84pt off 2.51pt}]  (262,94) -- (403.33,95) ;
\draw  [draw opacity=0][pattern=_ap7d8q7tq,pattern size=3.75pt,pattern thickness=0.75pt,pattern radius=0pt, pattern color={rgb, 255:red, 0; green, 0; blue, 0}] (403.33,38) -- (319.33,82) -- (262,37) -- cycle ;
\draw  [dash pattern={on 0.84pt off 2.51pt}]  (290.67,116.5) -- (361.33,117) ;
\draw    (290.67,59.5) -- (290.67,116.5) ;
\draw    (361.33,60) -- (361.33,117) ;
\draw    (290.67,59.5) -- (361.33,60) ;
\draw  [draw opacity=0][pattern=_m4ws0y35t,pattern size=3.75pt,pattern thickness=0.75pt,pattern radius=0pt, pattern color={rgb, 255:red, 0; green, 0; blue, 0}] (361.33,60) -- (361.33,117) -- (290.67,116.5) -- (290.67,59.5) -- cycle ;

\draw   (447,54.3) -- (483.3,18) -- (570.33,18) -- (570.33,102.7) -- (534.03,139) -- (447,139) -- cycle ; \draw   (570.33,18) -- (534.03,54.3) -- (447,54.3) ; \draw   (534.03,54.3) -- (534.03,139) ;
\draw  [dash pattern={on 0.84pt off 2.51pt}]  (447,139) -- (484.33,103) -- (570.33,102.7) ;
\draw  [dash pattern={on 0.84pt off 2.51pt}]  (483.3,18) -- (484.33,103) ;
\draw  [draw opacity=0][pattern=_rvu75omry,pattern size=3.75pt,pattern thickness=0.75pt,pattern radius=0pt, pattern color={rgb, 255:red, 0; green, 0; blue, 0}] (534.03,54.3) -- (534.03,139) -- (447,139) -- (447,54.3) -- cycle ;
\draw  [draw opacity=0][pattern=_n8caparkv,pattern size=3.75pt,pattern thickness=0.75pt,pattern radius=0pt, pattern color={rgb, 255:red, 0; green, 0; blue, 0}] (570.33,18) -- (534.03,54.3) -- (447,54.3) -- (483.3,18) -- cycle ;

\end{tikzpicture}
\end{center}
\caption{$\mathcal{B}_{I^3}=\Big{\{}\{1\}, \{2\}, \{3\}, \{4\}, \{1,2\}, \{1,2,3\}, \{1,2,3,4\} \Big{\}}$.}
\end{figure}}
It follows from Lemma \ref{lemma} that $\EuScript{P}_{6}$ can be obtained from $I^{3}$ by two truncations corresponding to some decompositions. The last step is to truncate some vertex. This means that $\mathcal{B}\backslash \mathcal{B}_{I^{3}} = \big{\{}S^{1},S^{2}\big{\}}$ and $S^{2}=S_{1}^{2}\sqcup S^{2}_{2}\sqcup S^{2}_{3}$ has three disjoint elements of $\mathcal{B}_{I^{3}}$. Due to the minimality of decompositions, $\big{\{}S^{2}_{1}, S^{2}_{2}, S^{2}_{3}\big{\}}$ is $$\text{either}\quad \Big{\{} \{1\}, \{3\}, \{4\} \Big{\}} \quad \text{or} \quad \Big{\{} \{2\}, \{3\}, \{4\} \Big{\}}.$$
If the decomposition of $S^{1}$ consists of only two elements, then there will be only one triangular face in $\EuScript{P}_{6}$. Therefore $\big{\{}S^{1}, S^{2} \big{\}} = \Big{\{}\{1,3,4\}, \{2,3,4\} \Big{\}}$, and we arrive to $\EuScript{P}_{2}$, which is not combinatorially equivalent to $\EuScript{P}_{6}$ by the classification of the Bier polytopes.

\emph{Case 2}. $\mathcal{B}$ does not contain a building set of the $3$-cube $I^{3}$. Again, by Lemma \ref{lemma}, $\mathcal{B}\backslash \mathcal{B}_{\Delta^{3}} = \big{\{}S^{1}, S^{2}, S^{3}, S^{4} \big{\}}$ and $S^{4}$ corresponds to vertex truncation. Hence $S^{4}$ is decomposed as $S^{4}_{1}\sqcup S^{4}_{2} \sqcup S^{4}_{3}$, where $S^{4}_{i} \in \mathcal{B}_{\Delta^{3}}$, $i=1,2,3$. Observe that all $S^{4}_{i}$ are singletons. Therefore, if we add $S^{4}$ to $\mathcal{B}_{\Delta^{3}}$, we obtain a building set corresponding to the triangular prism $\Delta^{2}\times [0,1]$. Without loss of generality, let 
$$\mathcal{B}_{\Delta^{2}\times [0,1]} = \Big{\{}\{1\}, \{2\}, \{3\}, \{4\}, \{1,2,3\}, \{1,2,3,4\} \Big{\}}.$$
Since $\EuScript{P}_{6}$ has two triangular faces, we have to perform a vertex truncation once more. We arrive to $\EuScript{P}_{11}$. Due to symmetry, without loss of generality, assume that $$\mathcal{B}_{\EuScript{P}_{11}} = \Big{\{}\{1\}, \{2\}, \{3\}, \{4\}, \{1,2,3\}, \{2,3,4\}, \{1,2,3,4\} \Big{\}}.$$ If we cut off a vertex again, a hexagonal face will be formed (\hyperlink{Figure 2}{Figure 2}). By assumption, we can't add $\{1,2\}, \{1,3\}, \{1,4\}, \{2,3\}, \{2,4\}, \{3,4\}$ to $\mathcal{B}_{\EuScript{P}_{11}}$, a contradiction. Q.E.D.
\hypertarget{Figure 2}{
\begin{figure}[ht]
\begin{center}

\tikzset{
pattern size/.store in=\mcSize, 
pattern size = 5pt,
pattern thickness/.store in=\mcThickness, 
pattern thickness = 0.3pt,
pattern radius/.store in=\mcRadius, 
pattern radius = 1pt}
\usetikzlibrary{patterns}
\makeatletter
\pgfutil@ifundefined{pgf@pattern@name@_8j4jo5n5u}{
\pgfdeclarepatternformonly[\mcThickness,\mcSize]{_8j4jo5n5u}
{\pgfqpoint{0pt}{0pt}}
{\pgfpoint{\mcSize+\mcThickness}{\mcSize+\mcThickness}}
{\pgfpoint{\mcSize}{\mcSize}}
{
\pgfsetcolor{\tikz@pattern@color}
\pgfsetlinewidth{\mcThickness}
\pgfpathmoveto{\pgfqpoint{0pt}{0pt}}
\pgfpathlineto{\pgfpoint{\mcSize+\mcThickness}{\mcSize+\mcThickness}}
\pgfusepath{stroke}
}}
\makeatother

 \usetikzlibrary{patterns}
\tikzset{
pattern size/.store in=\mcSize, 
pattern size = 5pt,
pattern thickness/.store in=\mcThickness, 
pattern thickness = 0.3pt,
pattern radius/.store in=\mcRadius, 
pattern radius = 1pt}\makeatletter
\pgfutil@ifundefined{pgf@pattern@name@_sp96cluye}{
\pgfdeclarepatternformonly[\mcThickness,\mcSize]{_sp96cluye}
{\pgfqpoint{-\mcThickness}{-\mcThickness}}
{\pgfpoint{\mcSize}{\mcSize}}
{\pgfpoint{\mcSize}{\mcSize}}
{\pgfsetcolor{\tikz@pattern@color}
\pgfsetlinewidth{\mcThickness}
\pgfpathmoveto{\pgfpointorigin}
\pgfpathlineto{\pgfpoint{\mcSize}{0}}
\pgfpathmoveto{\pgfpointorigin}
\pgfpathlineto{\pgfpoint{0}{\mcSize}}
\pgfusepath{stroke}}}
\makeatother
\usetikzlibrary{patterns}
\usetikzlibrary{patterns}
\tikzset{
pattern size/.store in=\mcSize, 
pattern size = 5pt,
pattern thickness/.store in=\mcThickness, 
pattern thickness = 0.3pt,
pattern radius/.store in=\mcRadius, 
pattern radius = 1pt}
\makeatletter
\pgfutil@ifundefined{pgf@pattern@name@_g9vhhpj4j}{
\makeatletter
\pgfdeclarepatternformonly[\mcRadius,\mcThickness,\mcSize]{_g9vhhpj4j}
{\pgfpoint{-0.5*\mcSize}{-0.5*\mcSize}}
{\pgfpoint{0.5*\mcSize}{0.5*\mcSize}}
{\pgfpoint{\mcSize}{\mcSize}}
{
\pgfsetcolor{\tikz@pattern@color}
\pgfsetlinewidth{\mcThickness}
\pgfpathcircle\pgfpointorigin{\mcRadius}
\pgfusepath{stroke}
}}
\makeatother
\tikzset{every picture/.style={line width=0.75pt}} 

\begin{tikzpicture}[x=0.75pt,y=0.75pt,yscale=-1,xscale=1]
\usetikzlibrary{patterns}
\usetikzlibrary{patterns}
\draw  [draw opacity=0][pattern=_8j4jo5n5u,pattern size=3.75pt,pattern thickness=0.75pt,pattern radius=0pt, pattern color={rgb, 255:red, 0; green, 0; blue, 0}] (146.83,64.5) -- (104.33,78.5) -- (64.83,60.5) -- cycle ;
\draw   (177.33,118) -- (92.33,146) -- (13.33,110) -- (116.33,11) -- cycle ;
\draw    (116.33,11) -- (92.33,146) ;
\draw  [dash pattern={on 0.84pt off 2.51pt}]  (13.33,110) -- (177.33,118) ;
\draw    (64.83,60.5) -- (104.33,78.5) ;
\draw    (104.33,78.5) -- (146.83,64.5) ;
\draw  [dash pattern={on 0.84pt off 2.51pt}]  (64.83,60.5) -- (146.83,64.5) ;
\draw    (61.33,132.33) -- (101.33,95.33) -- (122.33,135.33) ;
\draw  [dash pattern={on 0.84pt off 2.51pt}]  (61.33,132.33) -- (122.33,135.33) ;
\draw  [dash pattern={on 0.84pt off 2.51pt}]  (45.33,125) -- (76.33,111.67) -- (48.33,76.67) ;
\draw    (48.33,76.67) -- (45.33,125) ;
\draw  [draw opacity=0][pattern=_sp96cluye,pattern size=4.275pt,pattern thickness=0.75pt,pattern radius=0pt, pattern color={rgb, 255:red, 0; green, 0; blue, 0}] (76.33,111.67) -- (45.33,125) -- (48.33,76.67) -- cycle ;
\draw  [draw opacity=0][pattern=_g9vhhpj4j,pattern size=4.125pt,pattern thickness=0.75pt,pattern radius=0.75pt, pattern color={rgb, 255:red, 0; green, 0; blue, 0}] (122.33,135.33) -- (101.33,95.33) -- (61.33,132.33) -- cycle ;
\draw  [line width=1.5]  (48.33,76.67) -- (64.83,60.5) -- (104.33,78.5) -- (101.33,95.33) -- (61.33,132.33) -- (45.33,125) -- cycle ;

\end{tikzpicture}
\end{center}
\caption{Hexagonal face after $3$ vertex truncations.}
\end{figure}}
\end{proof}

However, $\EuScript{P}_{6}$ has a realization with a regular normal fan, as the following result shows.

\begin{proposition}\label{DelzantP6}
$\EuScript{P}_{6}$ admits a Delzant realization in $\mathbb{R}^{3}$.
\end{proposition}
\begin{proof}
Indeed, $\EuScript{P}_{6}$ is formed by truncating a vertex of the nestohedron $\EuScript{P}_{7}$. Thus, as the desired Delzant realization, we can take the Delzant realization of $\EuScript{P}_{7}$ and perform the required vertex cut by means of a plane in general position parallel to the single triangular face of $\EuScript{P}_{7}$. Q.E.D.
\end{proof}
Here is the main result of our classification of two-dimensional Bier spheres.

\begin{theorem}\label{MainClassTheo}\hypertarget{theorem 4}
There are exactly $13$ pairwisely non-isomorphic two-dimensional Bier spheres $\EuScript{S}_{i}$. The corresponding Bier polytopes $\EuScript{P}_{i}$ have the following properties:
\begin{enumerate}
\item Four of them are flag ($\EuScript{P}_{4},\EuScript{P}_{5},\EuScript{P}_{8},\EuScript{P}_{10}$) and the rest are not;
\item For each $i=1,\ldots,5,\not${$6$}$,7,\ldots,13$ the polytope $\EuScript{P}_{i}$ is a nestohedron;
\item Each of them admits a Delzant realization in $\mathbb{R}^{3}$, see \hyperlink{Table 2}{Table 2}.
\end{enumerate}
\end{theorem}
\begin{proof}
Indeed, it follows from the Lemma~\ref{lemma} that the following building sets $\mathcal{B}_{i}$ define the nestohedra realizations of the simple polytopes $\EuScript{P}_{i}$:
\footnotesize
\begin{itemize}
    \item $\mathcal{B}_{1}=\Big{\{}\{1\}, \{2\}, \{3\}, \{4\}, \{1,2,3,4\}, \{1,3,4\}, \{1,2,4\}, \{1,2,3\}, \{2,3,4\} \Big{\}}$;
    \item $\mathcal{B}_{2}=\Big{\{}\{1\}, \{2\}, \{3\}, \{4\}, \{1,2\}, \{1,2,3\}, \{1,2,3,4\}, \{1,3,4\}, \{2,3,4\} \Big{\}}$;
    \item $\mathcal{B}_{3}=\Big{\{}\{1\}, \{2\}, \{3\}, \{4\}, \{1,2\}, \{1,2,3\}, \{1,2,3,4\}, \{2,3\}, \{1,3,4\}\Big{\}}$;
    \item $\mathcal{B}_{4} = \Big{\{}\{1\}, \{2\}, \{3\}, \{4\}, \{1,2\}, \{1,2,3\}, \{1,2,3,4\}, \{3,4\}, \{1,3,4\} \Big{\}}$;
    \item $\mathcal{B}_{5}=\Big{\{}\{1\}, \{2\}, \{3\}, \{4\},\{1,2\}, \{1,2,3\}, \{1,2,3,4\}, \{3,4\}, \{1,2,4\}\Big{\}}$;
    \item $\mathcal{B}_{7}=\Big{\{}\{1\}, \{2\}, \{3\}, \{4\}, \{1,2\}, \{1,2,3\}, \{1,2,3,4\}, \{1,3,4\}\Big{\}}$;
    \item $\mathcal{B}_{8}=\Big{\{}\{1\}, \{2\}, \{3\}, \{4\}, \{1,2\}, \{1,2,3\}, \{1,2,3,4\}, \{1,3\}\Big{\}}$;
    \item $\mathcal{B}_{9}=\Big{\{}\{1\}, \{2\}, \{3\}, \{4\}, \{1,2,3,4\}, \{1,2,3\}, \{1,2,4\}, \{1,3,4\}\Big{\}}$;
    \item $\mathcal{B}_{10}=\Big{\{}\{1\}, \{2\}, \{3\}, \{4\}, \{1,2\}, \{1,2,3\}, \{1,2,3,4\} \Big{\}}$;
    \item $\mathcal{B}_{11} = \Big{\{}\{1\}, \{2\}, \{3\}, \{4\}, \{1,2,3\}, \{2,3,4\}, \{1,2,3,4\} \Big{\}}$;
    \item $\mathcal{B}_{12}=\Big{\{}\{1\}, \{2\}, \{3\}, \{4\}, \{1,2,3,4\}, \{1,2,3\} \Big{\}}$;
    \item $\mathcal{B}_{13}=\Big{\{}\{1\}, \{2\}, \{3\}, \{4\}, \{1,2,3,4\} \Big{\}}$.
\end{itemize}
\normalsize
($\mathcal{B}_{6}$ does not exist by Theorem~\ref{NotDelzantTheorem1}).

By Proposition \ref{DelzantP6}, the polytope $\EuScript{P}_{6}$ is not a nestohedron but admits a Delzant realization in $\R^{3}$.
\end{proof}


\section{General case: Buchstaber number}

In this section we turn to a discussion of the combinatorial properties of Bier spheres related to
the topological properties of their polyhedral products studied in the framework of toric topology.

The key notion of study in toric topology~\cite{BPAm, Buchstaber-Panov} is that of a polyhedral product. A particularly important case of a polyhedral product for us here is the following object.
\begin{definition}
The \emph{moment-angle-complex} of an $(n-1)$-dimensional simplicial complex $\sK$ on $[m]$ is a space
$$
\mathcal{Z}_{\mathcal{K}}=(\mathbb D^2, \mathbb S^1)^{\sK}:=\bigcup_{I\in \mathcal{K}}(\mathbb D^2, \mathbb S^1)^I\subseteq(\mathbb D^{2})^{m},
$$
where
$$
(\mathbb D^2, \mathbb S^1)^I:=\prod\limits_{i=1}^m\,X_i, \text{ for }X_i=\mathbb D^2, \text{ if }i\in I,\text{ and }X_i=\mathbb S^1, \text{ otherwise.}
$$

Here we denote by $\mathbb D^2=\{z\in\C\,|\,|z|\leq 1\}$ the unit disc in the complex plane and by $\mathbb S^1$ its boundary. Note that $\zk$ can be naturally viewed as an $(m+n)$-dimensional cellular subspace in the unitary polydisc $(\mathbb D^{2})^{m}$ in the complex Euclidean space $\C^m$.

Similarly one defines the {\emph{real moment-angle-complex}} of $\sK$ to be $\mathcal{R}_{\mathcal{K}}=(\mathbb D^1, \mathbb S^0)^{\sK}$. This is an $n$-dimensional cellular subspace in the unitary cube $[-1;+1]^n$ in the real Euclidean space $\R^n$.
\end{definition}

In case of a polytopal sphere $K_P=\partial P^*$ corresponding to a simple $n$-dimensional polytope $P$, its moment-angle-complex can be realized as a complete intersection of Hermitian quadrics in $\C^m$ and therefore it acquires a smooth structure. This smooth closed $(m+n)$-dimensional manifold is called the {\emph{moment-angle manifold}} of $P$ and is denoted by $\zp$. The {\emph{real moment-angle manifold}} $\rp$ is defined similarly.

Observe that the torus $T^m:=(\mathbb S^1)^m$ acts on $\zk$ for any simplicial complex $\sK$ by the restriction of the natural coordinatewise action of $T^m$ on $\C^m$. Similarly, the real torus $\Z_2^m$ acts on $\rk$.

\begin{definition}
We call the {\emph{Buchstaber number}} of $\sK$ the maximal rank $s(\sK)$ of a toric subgroup in $T^m$ acting freely on $\zk$. Similarly, we call the {\emph{real Buchstaber number}} of $\sK$ the maximal rank $s_{\R}(\sK)$ of a real toric subgroup in $\Z_2^m$ acting freely on $\rk$.    
\end{definition}

It is well-known that for any $(n-1)$-dimensional simplicial complex $\sK$ on $[m]$, both $s(\sK)$ and $s_\R(\sK)$ are combinatorial invariants of $\sK$. The real Buchstaber number was introduced in~\cite{FM}, where it was proved that: 
$$
1\leq s(\sK)\leq s_\R(\sK)\leq m-n,
$$
see also~\cite[Proposition 3]{Ayz10} for an alternative proof.

\begin{lemma}
Let $\sK_1$ and $\sK_2$ be simplicial complexes on $[m]$ and suppose that their deleted join $\sK$ has dimension $m-2$. Then 
$$
s(\sK)=s_\R(\sK)=m+1.
$$
\end{lemma}
\begin{proof}
Throughout the proof we use the notation $\{1,2,\ldots,m\}$ for the vertices of $\sK_1$ and $\{1',2',\ldots,m'\}$ for the vertices of $\sK_2$.

Denote by $(e_1,\ldots,e_{m-1})$ the standard basis of the lattice $\Z^{m-1}$ and consider the following labelling of the vertices of $\sK_1$ and $\sK_2$:
$$
i,i'\mapsto e_{i},\text{ for all }1\leq i\leq m-1.
$$
$$
m,m'\mapsto e_1+e_2+\ldots+e_{m-1},
$$

Let $\sigma$ be a maximal simplex in $\sK$. Since $\dim \sK=m-2$, by definition, the number of vertices $|\sigma|$ of $\sigma$ is less or equal to $m-1$. 

Since $\sK$ is a deleted join, by definition, one and only one of the next three cases can occur:
\begin{enumerate}
\item $m\in\sigma$ and $m'\notin\sigma$;
\item $m'\in\sigma$ and $m\notin\sigma$;
\item $\{m,m'\}\cap\sigma=\varnothing$.
\end{enumerate}

For each of the first two cases of the three above, note that the labels of $\sigma$ form a part of the lattice basis in $\Z^{m-1}$, since the set of vectors:
$$
\{e_1+e_2+\ldots+e_{m-1},e_1,e_2,\ldots,\hat{e}_{i},\ldots,e_{m-1}\} 
$$
forms a lattice basis in $\Z^{m-1}$ for any $1\leq i\leq m-1$.

For the last one of the three cases above, note that the labels of $\sigma$ form a part of the lattice basis in $\Z^{m-1}$, since the set of vectors:
$$
\{e_1,e_2,\ldots,e_{m-1}\}
$$
forms a lattice basis in $\Z^{m-1}$.

Consider the linear map $\Lambda$ from $\Z^{2m}$ to $\Z^{m-1}$ induced by our labelling of $[m]\sqcup [m']$. It remains to observe that the kernel of the corresponding exponential map $\mathrm{exp}(\Lambda)$ between the corresponding (real) tori is a (real) toric subgroup $H(\Lambda)$ of dimension $2m-(m-1)=m+1$ in the $2m$-dimensional coordinate (real) torus, which acts freely on the (real) moment-angle-complex of $\sK$. This finishes the proof.
\end{proof}

The next result shows that both the real and complex Buchstaber numbers of a Bier sphere are always maximal possible.

\begin{theorem}\label{BuchNumTheo}
For any simplicial complex $\sK\neq\Delta_{[m]}$ with $m\geq 2$ vertices (including the ghost ones) we get:
$$
s(\B(\sK))=s_\R(\B(\sK))=m+1.
$$
\end{theorem}
\begin{proof}
Set $\sK_1:=\sK$ and $\sK_2:=\widehat{\sK}$, its Alexander dual. Then their deleted join equals $\B(\sK)$, the Bier sphere of $\sK$, which is a PL-sphere of dimension $m-2$. Then, by the previous lemma, $\B(\sK)$ has the maximal possible (real) Buchstaber number. Q.E.D.    
\end{proof}


\section{Examples, applications and open problems}

In this section we consider some examples of applications of our main results and formulate
some open problems. Let us begin with the following example.

\begin{example}
Let $\sK$ be the $\ell$-skeleton of the simplex $\Delta_{[m]}$, that is: $\sK=\mathrm{sk}^{\ell}(\Delta^{m-1})$, where $0\leq\ell\leq m-3$. Then its Alexander dual is $\widehat{\sK}=\mathrm{sk}^{m-\ell-3}(\Delta^{m-1})$. 

Applying Theorem~\ref{BuchNumTheo}, we get that $\B(\sK)$ has $2m$ geometrical ('real') vertices and the Buchstaber numbers $s(\B(\sK))=s_\R(\B(\sK))=m+1$, for all $m\geq 3$ and any $0\leq\ell\leq m-3$.    
\end{example}

If $P$ is a simple $n$-dimensional polytope with $m$ facets having Buchstaber number $s(P):=s(K_P)=m-n$, then a {\emph{characteristic map}} $\Lambda\colon\Z^m\to\Z^n$ exists; that is, each set of columns of $\Lambda$ corresponding to a set of facets of $P$ intersecting at a vertex of $P$ forms an integer basis of $\Z^n$. Therefore, if we denote by $H(\Lambda)$ the kernel of the exponential map $\mathrm{exp}(\Lambda)$, then the $(m-n)$-dimensional torus $H(\Lambda)$ acts smoothly and freely on the moment-angle manifold $\zp$ of $P$. 

\begin{definition}
Let $P$ be an $n$-dimensional simple polytope with $m$ facets. In case $s(P)=m-n$ the quotient space $\zp/H(\Lambda)$ is a $2n$-dimensional smooth manifold called a {\emph{quasitoric manifold}} of $P$ and denoted by $M(P,\Lambda)$. In case $s_\R(P)=m-n$, its real analogue, an $n$-dimensional smooth manifold $M_\R(P,\Lambda_\R)$ called a {\emph{small cover}} of $P$ is defined in a similar way. It is well-known that the canonical quasitoric manifold (respectively, small cover) of a Delzant polytope is a smooth projective toric variety over $\mathbb{C}$ (respectively, over $\mathbb{R}$).
\end{definition}

Each Delzant polytope $P$ has a regular normal fan and therefore gives rise to a {\emph{canonical
characteristic map}} $\Lambda_{P}$: the columns of the integer matrix of the map $\Lambda_P$ are the primitive normal vectors to facets of $P$. We denote the corresponding quasitoric manifold by $\EuScript{M}=M(P,\Lambda_{P})$. The reduction of the matrix $\Lambda_{P}$ modulo $2$ yields a {\emph{canonical characteristic map}} $\overline{\Lambda}_{P}$ and the corresponding small cover $\overline{\EuScript{M}}=M_\R(P,\overline{\Lambda}_{P})$. 

Applying Theorem~\ref{MainClassTheo}, we get $13$ characteristic maps $\Lambda_{\EuScript{P}_i}$, quasitoric manifolds $\EuScript{M}_{i}$ and small covers $\overline{\EuScript{M}_{i}}$. Note that for $i\ne 6$ they are canonical for the canonical Delzant realizations of $\EuScript{P}_{i}$, described above. In the case of $i=6$ we take the realization from Proposition \ref{DelzantP6}. 

In~\cite[Theorem 4.5.2]{Fenn} the following important result was obtained.

\begin{theorem}[{\cite{Fenn}}]
Let $P$ be a nestohedron with a connected building set $\mathcal{B}$. The column
of $\Lambda_{P}$ corresponding to a facet labelled by a set $I \in  \mathcal{B} \backslash \mathcal{B}_{\max}$ is equal to $(v_1,\ldots,v_n)^t$, where
\begin{equation*}
v_{i} = 
 \begin{cases}
   1 \quad i\in I,\text{ }n+1\notin I;\\
   0 \quad i\in I,\text{ }n+1\in I;\\
   0 \quad i\notin I,\text{ }n+1\notin I;\\
   -1 \quad i\notin I,\text{ }n+1\in I.\\
 \end{cases}
\end{equation*}
\end{theorem}

This result and Theorem~\ref{MainClassTheo} provide explicit forms of the canonical characteristic matrices $\Lambda_{\EuScript{P}_{i}}$ for all $1\leq i\leq 13$. We refer to \hyperlink{Appendix B}{Appendix B}, where we collect all those matrices.

The next two examples are concerned with the topology of manifolds $\EuScript{M}_{i}$ and $\overline{\EuScript{M}_{i}}$. For the first one we need the following classical result~\cite[Theorem 4.14]{Davis-Janiszkiewicz}.

\begin{theorem}[\cite{Davis-Janiszkiewicz}]
Let $\EuScript{M} = M(P, \Lambda)$ be a quasitoric manifold with $\Lambda=(\lambda_{ij})$, $1\leq i\leq n$, $1\leq j\leq m$. Then the integer cohomology ring of $\EuScript{M}$ is given by
$$
H^{*}(\EuScript{M})\cong \mathbb{Z}[v_{1},\ldots,v_{m}]/\mathcal{I},
$$
where $v_{i}\in H^{2}(\EuScript{M})$ is the class Poincar\'e dual to the characteristic submanifold $\EuScript{M}_{i}$ and $\mathcal{I}$ is the ideal generated by elements of the following two types:
\begin{enumerate}
    \item the square-free monomials $v_{i_{1}}\ldots v_{i_{k}}$, whenever $\EuScript{M}_{i_{1}}\cap \ldots \cap \EuScript{M}_{i_{k}}=\varnothing$ (the Stanley--Reisner relations);
    \item the linear forms $t_{i}=\lambda_{i1}v_{1}+\ldots+\lambda_{im}v_{m}$, $1\leq i\leq n$.
\end{enumerate}
\end{theorem}
The statement of the last theorem for projective toric varieties is known as the Danilov-Jurkiewicz theorem, see~\cite[Theorem 10.8]{Danilov}.
\begin{example}\label{QuasitoricEx}
It can be deduced from the previous result or showed from the cellular structure on a quasitoric manifold directly, that the integer homology groups of a quasitoric manifold are free abelian, they can be non-zero only in even dimensions and the Betti numbers $\beta_{2k}(M(P,\Lambda))=h_k(P)$ for all $0\leq k\leq n$. 

We compute the $f$-vectors of all the 13 Bier spheres $\EuScript{S}_{i}$, see \hyperlink{Appendix A}{Appendix A}. Using them, one can immediately restore the corresponding $h$-vectors and show that the six-dimensional quasitoric
manifolds $\EuScript{M}_{i}$ for $1 \le i \le 13$ are nonsingular projective toric varieties with the integer homology groups shown in Table 3.  

\begin{center}
\begin{table}[h!]
\resizebox{12cm}{!}{
\begin{tabular}{ | c | c | c | c | c | }
\hline
 $H_{0}$ & $H_{2}$ & $H_{4}$ & $H_{6}$ & $\EuScript{M}_{i}$ \\ \hline
 $\mathbb{Z}$ & $\mathbb{Z}^{5}$ & $\mathbb{Z}^{5}$ & $\mathbb{Z}$ & $i=1,2,3,4,5,6$ \\ \hline 
 $\mathbb{Z}$ & $\mathbb{Z}^{4}$ & $\mathbb{Z}^{4}$ & $\mathbb{Z}$ & $i=7,8,9$ \\ \hline
 $\mathbb{Z}$ & $\mathbb{Z}^{3}$ & $\mathbb{Z}^{3}$ & $\mathbb{Z}$ & $i=10,11$ \\ \hline 
 $\mathbb{Z}$ & $\mathbb{Z}^{2}$ & $\mathbb{Z}^{2}$ & $\mathbb{Z}$ & $i=12$\\ \hline 
 $\mathbb{Z}$ & $\mathbb{Z}$ & $\mathbb{Z}$ & $\mathbb{Z}$ & $i=13$ \\ \hline 
\end{tabular}}
\caption{Integral homology groups of the quasitoric manifolds $\EuScript{M}_{i}$}
\end{table}
\end{center}

\end{example}

We discuss the orientability of the small covers over the two-dimensional Bier spheres in the following example.

\begin{example}\label{SmallCoverEx}
Using the criterion~\cite[Theorem 1.7]{Nakayama-Nishimura}, one can easily check the orientability of the small covers $\overline{\EuScript{M}_{i}}$ for $1\leq i\leq 13$. 

Indeed, by that criterion, the 3-dimensional smooth manifold $\overline{\EuScript{M}_{i}}$ is orientable if and only if there exists a basis $\{e_{1},e_{2},e_{3}\}$ of the space $\mathbb{Z}^{3}_{2}$ such that the image of $\overline{\Lambda}_{\EuScript{P}_{i}}$ is contained in the set $\{e_{1},e_{2},e_{3},e_{1}+e_{2}+e_{3}\}$.

Thus, the mod 2 reduction of the characteristic matrices collected in \hyperlink{Appendix B}{Appendix B} shows that among the $3$-dimensional small covers $\overline{\EuScript{M}_{i}}$, the 1st, 6th, 9th, 11th, 12th and 13th are orientable and the rest are not.
\end{example}

Finally, we state a few open problems relevant to the content of the results introduced above.

\begin{prob}
Find a simplicial complex $\sK$ such that its Bier sphere $\B(\sK)$ is non-polytopal. 
\end{prob}

\begin{prob}
Classify all the Bier spheres combinatorially equivalent to nerve complexes of generalized truncation polytopes.    
\end{prob}

\begin{prob}
Classify all the Bier spheres combinatorially equivalent to nerve complexes of nestohedra.    
\end{prob}

\newpage


\section{Appendix A}\hypertarget{Appendix A}
Here we collect the combinatorial data for two-dimensional Bier spheres.
\subsubsection*{$\mathrm{MF}$ table of the two-dimensional Bier spheres}
\begin{flushleft}
\resizebox{17cm}{!}{
\begin{tabular}{ | c | c | }
\hline
 $i$ & $\mathrm{MF}(\EuScript{S}_{i})$  \\ \hline
 $1$ & $x_{1}x_{2}, x_{1}x_{3}, x_{1}x_{4}, x_{2}x_{3}, x_{2}x_{4}, x_{3}x_{4}, y_{2}y_{3}y_{4}, y_{1}y_{3}y_{4}, y_{1}y_{2}y_{4},$$ $$ y_{1}y_{2}y_{3}, x_{1}y_{1},x_{2}y_{2}, x_{3}y_{3}, x_{4}y_{4}$ \\ \hline $2$ & $x_{3}x_{4}, x_{1}x_{3}, x_{1}x_{4}, x_{2}x_{4}, x_{2}x_{3}, y_{3}y_{4}, y_{1}y_{2}y_{4}, y_{1}y_{2}y_{3}, x_{1}y_{1}, x_{2}y_{2}, x_{3}y_{3}, x_{4}y_{4}$ \\ \hline $3$ &  $x_{3}x_{4}, x_{1}x_{3}, x_{1}x_{4}, x_{2}x_{4}, y_{3}y_{4}, y_{1}y_{4}, y_{1}y_{2}y_{3},$$ $$x_{1}y_{1}, x_{2}y_{2}, x_{3}y_{3}, x_{4}y_{4}$ \\ \hline $4$ & $x_{2}x_{3}, x_{1}x_{3}, x_{1}x_{4}, x_{2}x_{4}, y_{3}y_{4}, y_{1}y_{2}, x_{1}y_{1}, x_{2}y_{2}, x_{3}y_{3}, x_{4}y_{4}$ \\ \hline $5$ &  $x_{3}x_{4}, x_{2}x_{3}, x_{2}x_{4}, y_{3}y_{4}, y_{2}y_{4}, y_{2}y_{3}, x_{1}y_{1}, x_{2}y_{2}, x_{3}y_{3}, x_{4}y_{4}$ \\ \hline 
 $6$ & $x_{1}x_{2}x_{3}, x_{2}x_{4}, x_{3}x_{4}, x_{1}x_{4}, y_{3}y_{4}, y_{2}y_{4}, y_{1}y_{4}, y_{1}y_{2}y_{3},$$ $$x_{1}y_{1}, x_{2}y_{2}, x_{3}y_{3}, x_{4}y_{4}$ \\ \hline
 $7$ & $x_{1}x_{4}, x_{2}x_{4}, x_{3}x_{4},y_{4}, y_{1}y_{2}y_{3}, x_{1}y_{1}, x_{2}y_{2}, x_{3}y_{3}$ \\ \hline 
 $8$ & $x_{2}x_{4}, x_{3}x_{4}, y_{4}, y_{2}y_{3}, x_{1}y_{1}, x_{2}y_{2}, x_{3}y_{3}$
 \\ \hline $9$ &  $x_{1}x_{2}x_{4}, x_{2}x_{3}x_{4}, x_{1}x_{3}x_{4}, y_{4}, y_{2}y_{3}, y_{1}y_{3}, y_{1}y_{2}, $$ $$x_{1}y_{1}, x_{2}y_{2}, x_{3}y_{3}$ \\ \hline 
 $10$ & $x_{3}x_{4}, y_{4}, y_{3}, x_{1}y_{1}, x_{2}y_{2}$ \\ \hline  
 $11$ & $x_{1}x_{3}x_{4}, x_{2}x_{3}x_{4}, y_{4}, y_{3}, y_{1}y_{2}, x_{1}y_{1}, x_{2}y_{2}$ \\ \hline 
 $12$ & $x_{1}x_{3}x_{4}, y_{4}, y_{3}, y_{1}, x_{2}y_{2}$ \\ \hline  
 $13$ & $x_{1}x_{2}x_{3}x_{4},y_{4}, y_{3}, y_{1}, y_{2}$\\\hline
\end{tabular}
}
\end{flushleft}

\subsubsection*{$f$-vectors in the flag case.}
\begin{flushleft}
\resizebox{3.5cm}{!}{
\begin{tabular}{ | c | c | }
\hline
 $i$ & $f(\EuScript{S}_{i})$  \\ \hline
$4$ & $(8, 18, 12)$ \\ \hline $$5$$ & $(8, 18, 12)$ \\ \hline $8$ & $(7, 15, 10)$ \\ \hline $10$ &  $(6, 12, 8)$ \\ \hline
\end{tabular}
}
\end{flushleft}
\subsubsection*{$f$-vectors in the non-flag case.}
\begin{flushleft}
\resizebox{3.5cm}{!}{
\begin{tabular}{ | c | c | }
\hline
 $i$ & $f(\EuScript{S}_{i})$  \\ \hline
$1$ & $(8, 18, 12)$ \\ \hline $2$ & $(8, 18, 12)$  \\ \hline $3$ &  $(8,18,12)$ \\ \hline 
$6$ & $(8,18,12)$ \\ \hline
$7$ & $(7,15, 10)$
\\ \hline $9$ & $(7, 15, 10)$  \\ \hline $11$ &  $(6, 12, 8)$\\ \hline 
$12$ & $(5,9,6)$ \\ \hline $13$ & $(4, 6, 4)$ \\ \hline
\end{tabular}
}
\end{flushleft}

\newpage


\section{Appendix B}\hypertarget{Appendix B}
Here we collect the characteristic matrices $\Lambda_{\EuScript{P}_{i}}$ of $\EuScript{P}_{i}$ for $1\leq i\leq 13$. \footnotesize
\begin{enumerate}
    \item $\Lambda_{\EuScript{P}_{1}}=\begin{pmatrix}1 & 0 & 0 & -1& 0 & 0& 1& -1\\
    0 & 1 & 0 & -1 & -1 & 0 & 1 & 0\\
    0 & 0 & 1 & -1 & 0 & -1 & 1 & 0\end{pmatrix}\in \text{Mat}_{3\times 8}(\mathbb{Z});$
    \item
    $\Lambda_{\EuScript{P}_{2}}=\begin{pmatrix}
       1 & 0 & 0 & -1& 1 & 1& 0& -1\\
    0 & 1 & 0 & -1 & 1 & 1 & -1 & 0\\
    0 & 0 & 1 & -1 & 0 & 1 & 0 & 0
    \end{pmatrix}\in \mathrm{Mat}_{3\times 8}(\mathbb{Z});$
    \item
    $\Lambda_{\EuScript{P}_{3}}=\begin{pmatrix}
    1 & 0 & 0 & -1 & 1 & 1 & 0 & 0\\
    0 & 1 & 0 & -1 & 1 & 1 & 1 & -1\\
    0 & 0 & 1 & -1 & 0 & 1 & 1 & 0
    \end{pmatrix}\in \text{Mat}_{3\times 8}(\mathbb{Z});$
    \item
    $\Lambda_{\EuScript{P}_{4}}=\begin{pmatrix}
        1 & 0 & 0 & -1 & 1 & 1 & -1 & 0\\ 0 & 1 & 0 & -1 & 1 & 1 & -1 & -1\\ 0 & 0 & 1 & -1 & 0 & 1 & 0 & 0
    \end{pmatrix}\in \mathrm{Mat}_{3\times 8}(\mathbb{Z});$
    \item
    $\Lambda_{\EuScript{P}_{5}}=\begin{pmatrix}
        1 & 0 & 0 & -1 & 1 & 1 & -1 & 0\\
        0 & 1 & 0 & -1 & 1 & 1 & -1 & 0\\
        0 & 0 & 1 & -1 & 0 & 1 & 0 & -1  
    \end{pmatrix}\in \mathrm{Mat}_{3\times 8}(\mathbb{Z});$
    \item
    $\Lambda_{\EuScript{P}_{6}}=\begin{pmatrix}
    1 & 0 & 0 & -1 & 0 & 0 & 1 & -1\\
    0 & 1 & 0 & 0 & -1 & 0 & 1 & -1\\
    0 & 0 & 1 & 0 & 0 & -1 & 1 & -1
    \end{pmatrix}\in \text{Mat}_{3\times 8}(\mathbb{Z});$
    \item 
    $\Lambda_{\EuScript{P}_{7}}=\begin{pmatrix}
   1 & 0 & 0 & -1 & 1 & 1 & 0\\
    0 & 1 & 0 & -1 & 1 & 1 & -1\\
    0 & 0 & 1 & -1 & 0 & 1 & 0
    \end{pmatrix}\in \text{Mat}_{3\times 7}(\mathbb{Z});$
    \item
    $\Lambda_{\EuScript{P}_{8}}=\begin{pmatrix}
    1 & 0 & 0 & -1 & 1 & 1 & 1\\
    0 & 1 & 0 & -1 & 1 & 1 & 0\\
    0 & 0 & 1 & -1 & 0 & 1 & 1
    \end{pmatrix}\in \text{Mat}_{3\times 7}(\mathbb{Z});$
    \item
    $\Lambda_{\EuScript{P}_{9}}=\begin{pmatrix}
    1 & 0 & 0 & -1 & 1 & 0 & 0\\
    0 & 1 & 0 & -1 & 1 & 0 & -1\\
    0 & 0 & 1 & -1 & 1 & -1 & 0
    \end{pmatrix}\in \text{Mat}_{3\times 7}(\mathbb{Z});$
    \item
    $\Lambda_{\EuScript{P}_{10}}=\begin{pmatrix}
    1 & 0 & 0 & -1 & 1 & 1\\
    0 & 1 & 0 & -1 & 1 & 1\\
    0 & 0 & 1 & -1 & 0 & 1
    \end{pmatrix}\in \text{Mat}_{3\times 6}(\mathbb{Z});$
    \item
    $\Lambda_{\EuScript{P}_{11}}=\begin{pmatrix}
    1 & 0 & 0 & -1 & 1 & -1\\
    0 & 1 & 0 & -1 & 1 & 0\\
    0 & 0 & 1 & -1 & 1 & 0
    \end{pmatrix}\in \text{Mat}_{3\times 6}(\mathbb{Z});$
    \item
    $\Lambda_{\EuScript{P}_{12}}=\begin{pmatrix}
    1 & 0 & 0 & -1 & 1\\
    0 & 1 & 0 & -1 & 1\\
    0 & 0 & 1 & -1 & 1
    \end{pmatrix}\in \text{Mat}_{3\times 5}(\mathbb{Z});$
    \item
    $\Lambda_{\EuScript{P}_{13}}=\begin{pmatrix}
   1 & 0 & 0 & -1\\
    0 & 1 & 0 & -1\\
    0 & 0 & 1 & -1
    \end{pmatrix}\in \text{Mat}_{3\times 4}(\mathbb{Z});$
\end{enumerate}

\newpage 

\subsection*{Acknowledgements}
The authors are grateful to Anton Ayzenberg, Taras Panov, and Rade \v{Z}ivaljevi\'c for numerous fruitful discussions, valuable comments and suggestions. The authors also express their gratitude to Younghan Yoon and the anonymous referee for the comments and
suggestions that made it possible to improve the presentation.

Limonchenko was supported by the Serbian Ministry of Science, Innovations and Technological Development through the Mathematical Institute of the Serbian Academy of Sciences and Arts. Sergeev`s work was carried out within the project ``Mirror Laboratories'' of HSE University, Russian Federation.

\normalsize
\end{document}